\documentclass{amsart}

\usepackage{xargs}                      
\usepackage[pdftex,dvipsnames]{xcolor}  

\usepackage{todonotes}
\newcommandx{\inbar}[2][1=]{\todo[linecolor=ForestGreen,backgroundcolor=ForestGreen!25,bordercolor=ForestGreen,#1]{#2}}
\newcommandx{\sarah}[2][1=]{\todo[linecolor=Violet,backgroundcolor=Violet!25,bordercolor=Violet,#1]{#2}}

\usepackage{amsmath, amsthm, xypic, amssymb, enumerate, mathrsfs, eufrak, verbatim, graphicx, bbm, stackrel,MnSymbol,caption,subcaption}
\usepackage{color}
\usepackage[toc,page]{appendix}

\usepackage[mathscr]{euscript}
\usepackage[all,cmtip,2cell]{xy}
\xyoption{rotate}
\UseTwocells

\usepackage[all]{xy}
\usepackage{hyperref}
\usepackage{tikz}
\usepackage{tikz-3dplot}
\usetikzlibrary{decorations.pathreplacing}

\DeclareMathOperator{\col}{colim}

\DeclareMathOperator*{\hl}{holim}

\DeclareMathOperator{\Hom}{Hom}

\newcommand{\D}{\displaystyle}

\newcommand{\cC}{\mathcal{C}}
\newcommand{\sD}{\mathcal{D}}
\newcommand{\cE}{\mathcal{E}}

\newcommand{\cP}{\mathcal{P}}

\newcommand{\cX}{\mathcal{X}}

\newtheorem{thm}{Theorem}[section]

\newtheorem*{thm*}{Theorem}
\newtheorem{prop}[thm]{Proposition}
\newtheorem{cor}[thm]{Corollary}
\newtheorem{lemma}[thm]{Lemma}

\theoremstyle{definition}
\newtheorem{defn}[thm]{Definition}

\theoremstyle{definition}
\newtheorem{ex}[thm]{Example}

\theoremstyle{definition}
\newtheorem{rem}[thm]{Remark}

\theoremstyle{definition}

\theoremstyle{definition}

\newtheorem*{rem*}{Remark}
\newtheorem*{exA*}{Example}

\newtheoremstyle{TheoremForIntro} 
        {.6em}{.6em}              
        {\itshape}                      
        {}                              
        {\bfseries}                     
        {. }                             
        { }                             
        {\thmname{#1}\thmnote{ \bfseries #3}}
    \theoremstyle{TheoremForIntro}
    \newtheorem{TheoremIntro}[thm]{Theorem}
    
    \newtheorem{CorollaryIntro}[thm]{Corollary}

\input xy
\xyoption{all}

\newcommand{\itop}{\mathsf{isvt\text{-}Top}}

\newcommand{\topp}{\mathsf{Top}}
\newcommand{\ilink}[1]{\mathsf{M}_{\mathbf{H}}({#1})}
\newcommand{\imap}{\mathsf{Map_{isvt}}}
\newcommand{\emap}{\mathsf{Map_{eqvt}}}

\newcommand{\map}{\mathsf{Map}}

\newcommand{\itimes}{\overset{\mathsf{isvt}}{\times}}
\DeclareMathOperator*{\iprod}{\overset{\mathsf{isvt}}{\prod}}
\newcommand{\icell}{\mathcal{I}\text{-}\mathsf{cell}}

\newcommand{\icof}{\mathcal{I}\text{-}\mathsf{cof}}

\newcommand{\iinj}{\mathcal{I}\text{-}\mathsf{inj}}
\newcommand{\jinj}{\mathcal{J}\text{-}\mathsf{inj}}

\newcommand{\ipathW}{W^{[0,1]}}
\newcommand{\isusp}{S_{\mathcal{U}}}
\newcommand{\iloop}{\Omega_{\mathcal{U}}}
\newcommand{\Tot}{\mathsf{Tot}}
\DeclareMathOperator{\crep}{crep}

\title{An isovariant Blakers--Massey theorem}
\author{Inbar Klang and Sarah Yeakel}

\begin{document}

\begin{abstract}
     An isovariant map is an equivariant map between $G$-spaces which strictly preserves isotropy groups. In this paper, we lay the groundwork for the study of isovariant stable homotopy theory when $G$ is a finite group. We prove an isovariant Blakers--Massey theorem and its $n$-cubical generalization, define a suitable notion of suspension (by a trivial representation sphere) in the isovariant category, and prove an isovariant Freudenthal suspension theorem.

\end{abstract}

\maketitle

\section{Introduction}
Let $G$ be a finite group. If $X$ and $Y$ are compactly generated spaces with continuous left $G$-actions, an \emph{isovariant} map is an equivariant continuous function $f:X \to Y$ such that $G_x = G_{f(x)}$ for all $x \in X$. That is, an isovariant map preserves the $G$-action and strictly preserves isotropy subgroups.

\medskip

For example, consider the cyclic group with two elements, $C_2$. The one-point space $\ast$ has a trivial $C_2$-action and the unit disk $D^2$ can be given the $C_2$-action which reflects across the vertical axis. Maps from $\ast$ to $D^2$ with image in the $C_2$-fixed points are both equivariant and isovariant. Indeed, any equivariant map which is injective is also isovariant. The map $D^2 \to \ast$ is equivariant, but not isovariant.

\medskip

Questions of isovariant homotopy theory often arise when studying smooth manifolds with symmetries, for example, in equivariant surgery theory \cite{browderquinn}, in equivariant $h$-cobordism theory \cite{Luc,MMhcob}, and even in the classification of $G$-spaces \cite{palais}. In previous papers, the authors study isovariant homotopy theory, proving an isovariant Elmendorf's theorem (\cite{sayisvtelm}, Theorem 3.1), an isovariant Whitehead's theorem (\cite{KYfixed}, Theorem 3.10), and providing a complete obstruction for eliminating fixed points of isovariant maps up to isovariant homotopy (\cite{KYfixed}, Theorem 4.1). In this paper, we lay the groundwork for the study of isovariant stable homotopy theory.

\medskip

Two of the central fundamental results leading to the development of stable homotopy theory were the Blakers--Massey theorem \cite{BM} and the Freudenthal suspension theorem \cite{Fsusp}. In a stable category, homotopy pushout squares and homotopy pullback squares coincide; this is not true in the category of spaces. In modern terms, the Blakers--Massey theorem gives a measure of how close a homotopy pushout square is to being a homotopy pullback square, based on connectivity of the maps. For a concise history of the result, see Section 4.1 of \cite{munsonvolic}. More precisely,

\medskip

\textbf{Blakers--Massey theorem.} Suppose that the following is a homotopy pushout square of spaces:
\[\xymatrix{
Z \ar[r]^-v \ar[d]_-u & Y \ar[d]^-j \\
X \ar[r]^-i & W
}\]
If $u$ is $n$-connected and $v$ is $m$-connected, then the map from $Z$ to the homotopy pullback of $i$ and $j$ (also known as the cartesian gap map of the square) is $(n + m - 1)$-connected.

\medskip

Similarly, in a stable category, the suspension functor and the loop functor are inverses of each other. This again fails in the category of spaces, but the Freudenthal suspension theorem, which follows from the Blakers--Massey theorem, states that for a highly connected space, the suspension and loop functors are close to being inverses. More precisely, the theorem states that if $X$ is an $n$-connected cell complex, then the map $X \to \Omega \Sigma X$ is $(2n+1)$-connected. The Blakers--Massey theorem and the Freudenthal suspension theorem allowed for the development of the stable homotopy category and for its implementation using the category of spectra.

\medskip

Versions of the Blakers--Massey and Freudenthal suspension theorems hold in the equivariant setting as well. Hauschild first proved an equivariant Blakers--Massey with conclusions on representation-graded homotopy groups \cite[2.1]{Hau}, and Dotto proved a higher-dimensional genuine version which includes $G$-actions on the diagrams \cite[2.1.3]{Dot}. Equivariant suspension theorems were proven by Hauschild (\cite[2.2, 2.3]{Hau}), Namboodiri (\cite[2.3]{namFsusp}, recovered by Dotto in \cite[2.2.1]{Dot}), and Lewis (\cite[5.1]{lewis}). Many of these theorems involve suspension by representation spheres, which we do not define (yet) in the isovariant context. 

\medskip

In this paper, we prove an isovariant Blakers--Massey theorem (and its $n$-cubical generalization) for an isovariant notion of connectivity, analogous to equivariant connectivity, but using mapping spaces of isovariant simplices instead of mapping spaces of orbits.

\begin{TheoremIntro}[\ref{thm-IBM}, (\ref{thm-higher-isvt-BM})]
Suppose that the following is a homotopy pushout square in the isovariant category:
    \[\xymatrix{
Z \ar[r]^-v \ar[d]_-u & Y \ar[d]^-j \\
X \ar[r]^-i & W
}\]
If $u$ is $n_\bullet$-connected and $v$ is $m_\bullet$-connected, then the cartesian gap map of the square is $(n_\bullet + m_\bullet - 1)$-connected. Here $n_\bullet$ is a function on strictly increasing chains of subgroups of $G$, which are the isovariant analogues of orbits $G/H$.
\end{TheoremIntro}

The proof of this theorem relies on isovariant homotopy theory tools developed in \cite{sayisvtelm} and \cite{KYfixed}. An alternate proof would proceed by replacing the isovariant category with presheaves over the link orbit category, as in the isovariant Elmendorf theorem of \cite{sayisvtelm}. In this paper, we give a more concrete proof that remains in the category of isovariant spaces and uses explicit models for homotopy pushout and pullback in that category.

\medskip

As an application of the isovariant Blakers--Massey theorem, we prove an isovariant Freudenthal suspension theorem, although defining suspension in the isovariant category is far from straightforward.

\medskip

Suspension in topological spaces can be defined as a homotopy pushout along maps to the terminal object, the 1-point space.
\[\xymatrix{
X \ar[r] \ar[d] & {*} \\
{*}
}\]
Isovariant maps to the 1-point space are quite rare. In fact, no $G$-space is a terminal object in the category of $G$-spaces with isovariant maps, and therefore the papers \cite{sayisvtelm} and \cite{KYfixed} make use of a formal terminal object. 
We use $\itop$ and the term \textit{the isovariant category} to refer to the category of $G$-spaces and isovariant maps with a formal terminal object. 
However, any pushout in $\itop$ involving the formal terminal object results in the formal terminal object. For a useful definition of suspension, we make use of a homotopy terminal object. We identify a particular model for our homotopy terminal object: a complete $G$-universe $\mathcal{U}$, which is a countable direct sum of copies of the regular representation of $G$. We prove:

\begin{TheoremIntro}[\ref{prop-h-terminal}]
    The complete $G$-universe $\mathcal{U}$ is isovariantly weakly contractible.
\end{TheoremIntro}

As a corollary, we obtain that every isovariant cell complex admits a unique-up-to-isovariant-homotopy isovariant map to $\mathcal{U}$. We are then able to define isovariant suspension $\isusp X$ of an isovariant cell complex $X$ as a homotopy pushout along maps to $\mathcal{U}$.
\[\xymatrix{
X \ar[r] \ar[d] & \mathcal{U} \\
\mathcal{U}
}\]
We are then also able to define an isovariant loop space functor, $\iloop$, as a homotopy pullback. Now equipped with useful notions of isovariant suspension and isovariant loop spaces, a straightforward application of our isovariant Blakers--Massey theorem (Theorem \ref{thm-IBM}) yields an isovariant Freudenthal suspension theorem:
\begin{CorollaryIntro}[\ref{thm-isvt-Freudenthal}]
    Let $X$ be an isovariantly $n_\bullet$-connected isovariant cell complex. Then the cartesian gap map $X \to \iloop \isusp X$ is isovariantly $(2n_\bullet + 1)$-connected.
\end{CorollaryIntro}

\begin{rem}
   Our notion of isovariant suspension corresponds to suspending by a trivial one-dimensional representation sphere. In equivariant stable homotopy theory, suspension by any finite-dimensional representation sphere plays a crucial role. For example, it is necessary to suspend by representation spheres and grade (co)homology on representations in order to prove an equivariant Poincaré duality theorem. In future work, we plan to study the possibility of isovariant suspension by representation spheres, and investigate its role in isovariant (stable) homotopy theory as well as connections to equivariant homotopy theory and equivariant manifold theory.  This may involve a generalization of our isovariant Blakers--Massey theorem, along the lines of Dotto's equivariant version \cite{Dot}.
\end{rem}

\subsection*{Structure of the paper.} In section 2, we review necessary background on isovariance and isovariant homotopy theory. In section 3, 
we give a model for homotopy limits and homotopy pushouts in the isovariant category.
In section 4, 
we reprove Hauschild's equivariant Blakers--Massey theorem \cite{Hau}, using techniques that apply to the isovariant case. We then prove the isovariant Blakers--Massey theorem and its $n$-cubical generalization. In section 5, we define the isovariant suspension $S_\mathcal{U}$ and prove the isovariant Freudenthal suspension theorem.

\subsection*{Acknowledgments.} We would like to thank Kate Ponto and Cary Malkiewich for many illuminating discussions on isovariance, and for their support during this project. We would like to thank Cary Malkiewich and Mona Merling for a useful discussion on Proposition 4.6 of their paper \cite{MMhcob}, which was instrumental in our proof of Theorem \ref{prop-h-terminal}. We would also like to thank Kristine Bauer for suggesting the possibility of an $n$-cubical isovariant Blakers--Massey theorem, and Boris Chorny for the alternative definition of isovariance. Thanks to the referee for suggestions that improved the flow of the paper. The authors would like to thank the Isaac Newton Institute for Mathematical Sciences, Cambridge, for support and hospitality during the programme \textbf{Equivariant homotopy theory in context}, where work on this paper was undertaken. This work was supported by EPSRC grant EP/Z000580/1.

\section{isovariant background}

We start by giving two alternative definitions of isovariance which may be of broader interest. 

Given an equivariant map $f:X \to Y$, there is a unique continuous map $f/G:X/G \to Y/G$ defined by taking the orbit $G\cdot x$ to $G \cdot f(x)$. We can form the following commutative square \[\xymatrix{ X \ar[r]^f \ar[d]_{\pi_X} & Y \ar[d]^{\pi_Y} \\ X/G \ar[r]^{f/G} & Y/G}\]
If the square is a pullback, then $f$ is isovariant. The converse holds if $X$ and $Y$ are further assumed to be weak Hausdorff, as proven in \cite[4.1]{isvtpullbacksquare}.

For another definition, consider the following double square of spaces.
\begin{gather}
    \begin{aligned}
\xymatrix{G/H \ar[r] \ar[d] & X \ar[r]^f \ar[d]_{\pi_X} & Y \ar[d]^{\pi_Y} \\ \ast \ar[r] & X/G \ar[r]^{f/G} & Y/G }
\end{aligned}
\label{def-alt-isvt}
\end{gather}

\begin{prop} Let $f:X \to Y$ be an equivariant map. 
The map $f$ is isovariant if and only if the left square of (\ref{def-alt-isvt}) being a pullback implies the outside square of (\ref{def-alt-isvt}) is a pullback for all maps $\ast \to X/G$ and for all subgroups $H$.
\end{prop}

\begin{proof} Suppose that the left square being a pullback implies the outside square is a pullback for all maps $\ast \to X/G$. Now let $x \in X$ and consider the square with lower horizontal map given by sending the one point to  $[x]=G \cdot x \in X/G$. The pullback of the left square is $G \cdot x \subset X$. By the orbit-stabilizer theorem, $G\cdot x \cong G/G_x$, so there is an $H \leq G$ with $G/H \cong G/G_x$, and the left square is a pullback. Then the outside square is also a pullback, so the orbit $G \cdot f(x)$ is isomorphic to $G/H$ as well, where $H=G_x$. Since $f$ is equivariant, we have that $G_x \subseteq G_{f(x)}$ and since $G_x \cong G_{f(x)}$ and $G$ is finite, $G_x=G_{f(x)}$.  

Conversely, assume $f$ is isovariant and the left square is a pullback for $[x]:\ast \to X/G$. Then the pullback $G \cdot x$ is homeomorphic to $G/H$ for some $H$. By the orbit-stabilizer theorem $H \cong G_x$, and since $f$ is isovariant, $G_x=G_{f(x)}$, so $H \cong G_{f(x)}$. Thus the outside square is also a pullback since $G \cdot f(x) \cong G/H$.
\end{proof} 

We now recall relevant background about isovariant homotopy theory from \cite{sayisvtelm} and \cite{KYfixed}. Since the category of $G$-spaces with isovariant maps does not have a terminal object, we study the category $\itop$ of $G$-spaces with isovariant maps and an additional formal terminal object.
In \cite{KYfixed}, we show that $\itop$ has a weak model structure in which $G$-manifolds are both fibrant and cofibrant, leading to an isovariant version of Whitehead's theorem.

\medskip

We assume familiarity with the basics of cofibrantly generated model structures as presented in \cite{hovey}. In particular, we use the following notation for classes of maps with certain lifting properties. Let $\mathcal{I}$ be a class of maps in a category $\mathcal{C}$ containing all small colimits. Then we define:

$\iinj$ is the class of maps with the right lifting property with respect to every map in $\mathcal{I}$.

$\icof$ is the class of maps with the left lifting property with respect to all maps in $\iinj$. 

$\icell$ is the class of relative $\mathcal{I}$-cell complexes. A relative $\mathcal{I}$-cell complex is a transfinite composition of pushouts of coproducts of elements of $\mathcal{I}$. Note that $\icell \subset \icof$, and in fact, $\icof$ consists of retracts of maps in $\icell$.

\medskip

Particular cases of $\mathcal{I}$ are of interest in isovariant homotopy theory. To define these, let $\Delta^n$ be the standard $n$-simplex in $\topp$, that is,
\[
\Delta^n = \left\{(t_0, \dots, t_n) \in [0,1]^{n+1} : \sum_{i=0}^n t_i =1 \right\}.
\]

\begin{defn} \label{defn:linkingsimplex}
   Let $\mathbf{H} = \{ H_0 < H_1 < ... < H_n \}$ be a chain of subgroups of $G$. The \emph{linking simplex} $\Delta_G^{\mathbf{H}}$ is the quotient of $G \times \Delta^n$ where $(g,x)\sim(g',x)$ if and only if $gH_{k}=g'H_{k}$, when $x=(t_0, \dots, t_{n-k}, 0, \dots, 0)$, $0 \leq k \leq n$ (that is, $x$ is in the standard $n-k$ simplex in $\Delta^n$). Let $G \times \Delta^n \to \Delta^{\mathbf{H}}_G$ be the natural projection and denote the image of $(g,x)\in G \times \Delta^n$ under the projection by $\langle g,x\rangle$. The space $\Delta^{\mathbf{H}}_G$ has a left $G$-action given by $g' \cdot \langle g,x \rangle=\langle g'g,x \rangle$; points of the form $\langle g,(t_0, \dots, t_{n-k},0, \dots,0)\rangle$ where $t_{n-k}\neq 0$ are fixed by $gH_kg^{-1}$ under the $G$-action. 
\end{defn}

\begin{rem}\label{ex-G/H}
    Intuitively, the quotient in the definition of isovariant linking simplex yields a $G$-space with orbit space consisting of $(n-i)$-dimensional simplices $G/H_i \times \Delta^{n-i}$, such that the inclusions of standard simplices $\Delta^{n-i-1} \subset \Delta^{n-i}$ are compatible with the inclusions $H_{i} < H_{i+1}$ (or $G$-conjugates thereof). If $\mathbf{H} = \{ H_0 \}$, then $\Delta_G^{\mathbf{H}} = G/H_0$. This will also be denoted $\Delta^{H_0}$. When $G$ is clear from context, we drop it from the notation.
\end{rem}

The following is a demonstration of the linking simplices of $C_6$, the cyclic group of order 6, along with isovariant (non-self) maps between them. The numbers in parentheses represent the number of distinct linear isovariant maps sending vertices to vertices.

\includegraphics[width=5in]{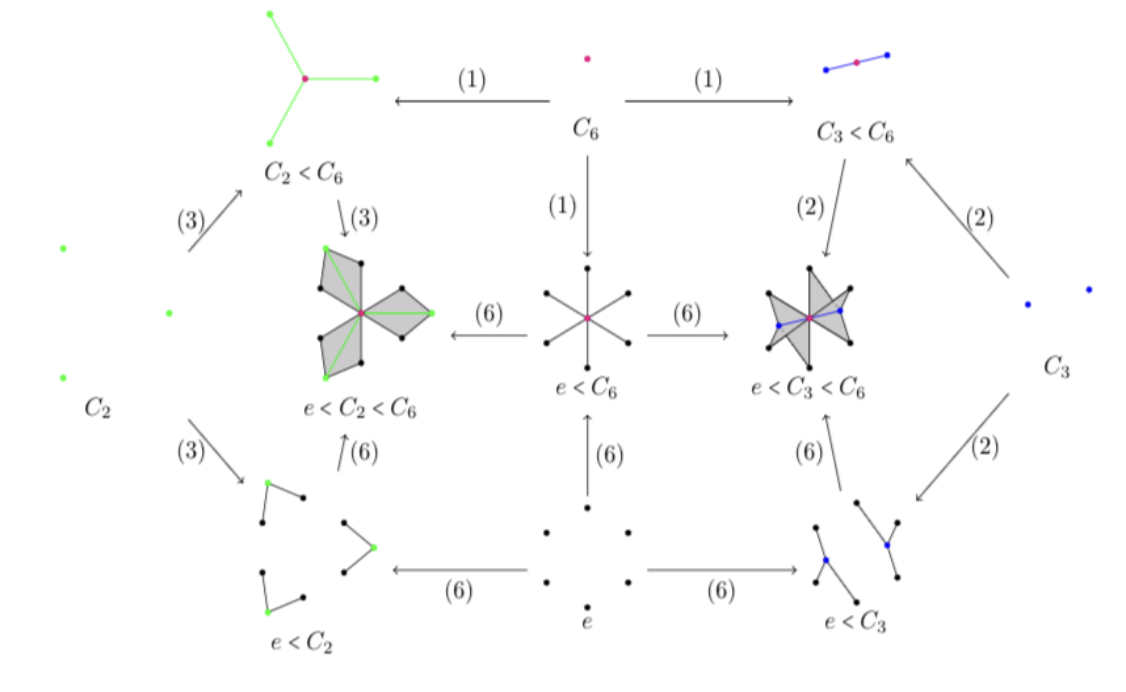}

We note that $\Delta^{H_0< \cdots < H_n}_G$ is the same as the ``equivariant simplex'' $\Delta_k(G;H_n, \dots, H_0)$ defined in \cite{illmanCW}, although in the equivariant simplex, subgroups may be repeated. Illman shows that the equivariant simplex is a compact Hausdorff space with orbit space $\Delta^n$.

The boundary of a linking simplex, $\partial \Delta^{\mathbf{H}}_G$, is the image of $G \times \partial \Delta^n$ (in the usual sense of $\partial \Delta^n$) under the identifications. 
Denote the boundary inclusion of a linking simplex by $b^\mathbf{H}: \partial \Delta^\mathbf{H} \to \Delta^\mathbf{H}$.

\begin{defn}\label{def-emap-imap}
    Let $X$ and $Y$ be $G$-spaces. Define $\emap(X,Y)$ to be the space of $G$-equivariant maps from $X$ to $Y$, and $\imap(X,Y)$ to be the subspace of $G$-isovariant maps from $X$ to $Y$. When $X$ is a linking simplex $\Delta^\mathbf{H}$, we denote the isovariant maps into $Y$ in the following ways: $\imap(\Delta^{\mathbf{H}},Y)=\ilink{Y}=Y_{\mathbf{H}}$. We refer to these as the isovariant $\mathbf{H}$-links of $Y$.
\end{defn}
Let $H \leq G$ be a subgroup. In the equivariant setting, $\emap(G/H,X)$ is equivalent to $X^H$, the subspace of elements which are $H$-fixed. In the isovariant setting, $\imap(\Delta^{H}_G,X)$ picks out the subspace of $X$ fixed by exactly $H$, denoted $X_H$. Here $\Delta^{H}_G=G/H$ denotes the linking simplex associated to the chain $\{H\}$ of length 0.

\medskip

The \emph{weak equivalences} of $\itop$ are the isovariant maps $f:X \to Y$ such that $\ilink{X} \to \ilink{Y}$  (that is, $\imap(\Delta^\mathbf{H},X) \to \imap(\Delta^\mathbf{H},Y)$) are weak equivalences of spaces for all strictly increasing chains $\mathbf{H}=\{H_0< \cdots< H_n\}$ of subgroups of $G$. We call these maps \textbf{isovariant weak equivalences}. These are, in general, different from isovariant homotopy equivalences. Recall that an \textbf{isovariant homotopy} between isovariant maps $f_0, f_1: X \to Y$ is an isovariant map $H: X \times [0,1] \to Y$ (where the product is the usual equivariant product of $G$-spaces, and $[0,1]$ has the trivial action) such that $H(x,0) = f_0(x)$ and $H(x,1) = f_1(x)$ for all $x \in X$. (We will soon see that $X \times [0,1]$ is an instance of the tensoring of $\itop$ over $\topp$.) An \textbf{isovariant homotopy equivalence} is an isovariant map $f$ which has an inverse up to isovariant homotopy.

\begin{defn}\label{def-weakly-contractible}
    We say that a $G$-space $X$ is \textit{isovariantly weakly contractible} if the unique map from $X$ to the formal terminal object is an isovariant weak equivalence. Equivalently, $X$ is isovariantly weakly contractible if and only if each $\ilink{X}$ is a weakly contractible space. We note that a $G$-space $X$ cannot be isovariantly (strictly) contractible in the usual sense because there are no isovariant maps from the formal terminal object to $X$.
\end{defn}

In order to define our classes of ``generating cofibrations" $\mathcal{I}$ and ``generating acyclic cofibrations" $\mathcal{J}$, we use pushout-products. Let $\cC, \sD,$ and $\cE$ be categories with all small colimits, and let $\otimes:\cC \times \sD \to \cE$ be a colimit-preserving functor. Then the pushout-product of a map $f:K \to X$ in $\cC$ and $g:L \to Y$ in $\sD$ is the map $f \Box g$ in $\cE$ from the pushout of the first three terms in the following square to the final vertex. This is also called the co-cartesian gap map of the square. 
\[
\xymatrix{
  K \otimes L \ar[r]^{id \otimes g} \ar[d]_{f \otimes id} & K \otimes Y \ar[d]^{f \otimes id} \\
  X \otimes L \ar[r]^{id \otimes g} & X \otimes Y
}
\]
The notion of pullback-hom will also be useful. If $\cC$ also has all small limits and there is a functor $\Hom_\sD:\sD^{op} \times \cE \to \cC$ with an adjunction between $-\otimes d$ and $\Hom_\sD(d,-)$ for every object $d \in \sD$, then the pullback-hom $\Hom_{\Box}(g,h)$ of $g:L \to Y$ in $\sD$ and $h:M \to Z$ in $\cE$ is the map in $\cC$ from the initial vertex to the pullback of the last three vertices of the square below. This is also called the cartesian gap map of the square.
\[
\xymatrix{
  \Hom(Y,M) \ar[r]^{g^\ast} \ar[d]_{h_\ast} & \Hom(L,M) \ar[d]^{h_\ast} \\
  \Hom(Y,Z) \ar[r]^{g^\ast} & \Hom(L,Z)
}
\]

An adjunction between pushout-products and pullback-homs is described in \cite[3.2.3]{caryparam}, which yields the following useful relationship between lifts. 

\begin{lemma} \cite[19.5]{rezkquasi} \label{lemma:popr-adj-pbhom} For maps as described above, a lift exists in the first diagram below if and only if a lift exists in the second diagram.
\[
\xymatrix{
  X \otimes L \coprod_{K \otimes L} K \otimes Y \ar[d]_{f \Box g} \ar[r] & M \ar[d]^h \\
  X \otimes Y \ar[r] \ar@{-->}[ur] & Z
}
 \, \, \, \, \, \, \, \, \ \ 
\xymatrix{
  K \ar[r] \ar[d]_f & \Hom(Y,M) \ar[d]^{\Hom_\Box(g,h)} \\
  X \ar[r] \ar@{-->}[ur] & \Hom(L,M)\times_{\Hom(L,Z)}\Hom(Y,Z)
}
\]
\end{lemma}

When $\cC$ is $\topp$, and $\sD=\cE=\itop$, the cartesian product defines a tensoring of $\itop$ over $\topp$
\[
\topp \times \itop \to \itop
\]
with right adjoint given by $\Hom_\sD =\imap$, see \cite[2.3]{sayisvtelm}.

When $\cC=\cE=\itop$ and $\sD=\topp$, we can define a cotensoring of $\itop$ over $\topp$. Given a space $A$ and an isovariant space $X$, $X^A \in \itop$ is the subspace of $\map(A,X)$ consisting of maps  $f$ whose image lands in a constant isotropy subspace, i.e. $G_{f(a)}=G_{f(a')}$ for all $a,a' \in A$. We give $X^A$ a pointwise $G$-action. We show that this is a cotensoring. That is,

\begin{prop}\label{prop-cotensoring}
    For a space $A$ and isovariant spaces $X,Y$, there is a natural isomorphism
\[
\map(A, \imap(X,Y)) \cong \imap(X, Y^A).
\]
\end{prop}

\begin{proof}
    This follows from two adjunctions. First, we have the adjunction coming from the tensoring above
    \[
    \map(A,\imap(X,Y)) \cong \imap(A \times X, Y).
    \]
    Second, we have an adjunction as described in Claim 2 of \cite[3.10]{KYfixed}:
    \[
    \imap(A \times X,Y) \cong \imap(X, Y^A).
    \]
 In more detail, an isovariant map $f: A \times X \to Y$ has $G_{(a,x)}=G_{f(a,x)}$ for all $a \in A$ and for all $x \in X$, and the $G$-action on $A \times X$ is given by $g \cdot (a,x)=(a, g \cdot x)$, so $G_{(a,x)}=G_x$ for all $a \in A$. Thus the function $f(-,x):A \to Y$ has constant isotropy (and the function $x \mapsto f(-,x)$ is isovariant).
\end{proof}

The underlying product of two $G$-spaces with the diagonal action is again a $G$-space, but does not necessarily have isovariant projection maps to the factors. Thus, cartesian product is not the categorical product in $\itop$.

\begin{defn}\label{def-iproduct}
    The isovariant product of $G$-spaces $X$ and $Y$ is the subspace of the cartesian product $X \times Y$ comprised of tuples of points $(x,y)$ such that $G_x = G_y$. We denote the isovariant product of two $G$-spaces by $X \itimes Y$.
\end{defn}

\begin{defn} \label{defn:geometricij} Let $s_{k-1}:S^{k-1} \to D^{k}$ be the usual boundary inclusion of the unit disk in $\topp$. Define $\mathcal{I}$ in $\itop$ as the class of pushout-products of $s_{k-1}$ with $b^\mathbf{H}$, that is,

\[
\mathcal{I}=\left\{s_{k-1} \Box b^\mathbf{H}: \left(S^{k-1} \times \Delta^{\mathbf{H}}_G \right) \cup_{S^{k-1}\times \partial \Delta_G^{\mathbf{H}}} \left(D^k \times \partial \Delta^{\mathbf{H}}_G \right) \to D^k \times \Delta^{\mathbf{H}}_G\right\}_{\mathbf{H},k}
\]
Define $\mathcal{J}$ in $\itop$ as pushout-products of the maps in $\mathcal{I} \subseteq \itop$ with $i_0:\{0\} \to [0,1]  \in \topp$, that is,
\[
\mathcal{J}=\left\{s_{k-1} \Box b^\mathbf{H} \Box i_0\right\}_{\mathbf{H},k}
\]

Let $\mathcal{W}$ be the class of isovariant weak equivalences, that is, the maps $f:X \to Y$ such that the induced map $\imap(\Delta^\mathbf{H},X) \to \imap(\Delta^{\mathbf{H}},Y)$ is a weak equivalence of spaces for all strictly increasing chains of subgroups $\mathbf{H}=\{H_0<\cdots<H_n\}$.
\end{defn} 

\begin{defn}\label{def-cof-fib}
    We say that a $G$-space $X$ is $\mathcal{I}$-cofibrant if the map $\emptyset \to X$ is in $\icof$. We say that a $G$-space $X$ is $\mathcal{J}$-fibrant if the map from $X$ to the formal terminal object is in $\jinj$. Being $\mathcal{J}$-fibrant is equivalent to the existence of a lift $B \to X$ in any diagram of the form 
    \[\xymatrix{
    A \ar[r] \ar[d]_-j & X \\
    B
    }\]
    for any $j \in \mathcal{J}$.
\end{defn}

In section 3 of \cite{KYfixed}, we prove an isovariant Whitehead's theorem. The theorem is stated for smooth $G$-manifolds, but the proof works equally well for $G$-spaces which are $\mathcal{I}$-cofibrant and $\mathcal{J}$-fibrant.

\begin{thm}\label{thm-isvt-whitehead}
    Suppose that $X$, $Y$ are $G$-spaces which are $\mathcal{I}$-cofibrant and $\mathcal{J}$-fibrant. Suppose that $f: X \to Y$ is an isovariant weak equivalence. Then $f: X \to Y$ is an isovariant homotopy equivalence.
\end{thm}

Finally, in order to define a notion of isovariant connectivity, we recall the notion of connectivity for topological spaces and maps between them.

\begin{defn} 
A space $X$ is $0$-connected if $\pi_0(X)$ is trivial, and for $k>0$,
a connected space $X$ is $k$-connected if $\pi_i(X, x_0)=0$ for all $x_0\in X$ and $0\leq i \leq k$. By convention, every space  is $(-2)$-connected, and
nonempty spaces are $(-1)$-connected.

A map of spaces is called $k$-connected if its homotopy fiber over each point is
$(k- 1)$-connected. Equivalently, $f:X \to Y$ is $k$-connected if $\pi_i f$ is an isomorphism for $i<k$ and a surjection for $i=k$.
\end{defn}

\section{Isovariant homotopy limits}

In this section we define homotopy limits and homotopy pushouts in the isovariant setting and describe their interaction with the isovariant mapping space functor, $\mathsf{M}_\mathbf{H}(-):\itop \to \topp$. We also define homotopy pushout and pullback squares, as well as the higher analogues of homotopy pushout squares, strongly cocartesian cubes.

We first give an explicit model for homotopy pullbacks in the isovariant category that we will use in the isovariant Blakers--Massey theorem.

\begin{defn}\label{def-pij}
    Let $W$ be a $G$-space. The cotensor $\ipathW$ is the $G$-space with pointwise action defined as
    \[
    \ipathW = \{ \gamma \in \map([0,1],W): G_{\gamma(0)}=G_{\gamma(t)} \forall t \in [0,1]\}.
    \]
    In \cite{KYfixed}, we denoted this space by $\mathsf{P}_{\mathsf{isvt}}W$.
\end{defn}

Our model for the homotopy pullback of $i: X \to W$, $j: Y \to W$ is then
\[P(i,j) = X \times_W \ipathW \times_W Y = \{ (x, \gamma, y) : \gamma(0) = i(x), \gamma(1) = j(y), G_{\gamma(0)} = G_{\gamma(t)} \forall t\}.\]
The isovariant maps $\pi_X: P(i,j) \to X$ and $\pi_Y: P(i,j) \to Y$ are given by projection onto the first and last component, respectively. 
In this definition, we can take the pullback over $W$ in the category of spaces or in the isovariant category, because for a $G$-space $W$, the pullback $X \itimes_W Y$ in $\itop$ agrees with the pullback $X \times_W Y$ in $\topp$. This is because $i: X \to W$ and $j: Y \to W$ are isovariant, so $i(x) = j(y)$ guarantees that $G_x = G_{i(x)} = G_{j(y)} = G_y$.

\medskip

This $P(i,j)$ is an instance of the Bousfield--Kan formula for a homotopy limit, which we elaborate on later in this section. First, we prove

\medskip

\begin{lemma}\label{lem-commutes-hopb}
For isovariant maps $i: X \to W$ and $j: Y \to W$, there is a canonical homeomorphism
    \[\ilink{P(i,j)} \cong \ilink{X} \times_{\ilink{W}} \map([0,1], \ilink{W}) \times_{\ilink{W}} \ilink{Y}\]
    for all chains of strictly increasing subgroups $\mathbf{H}$. That is, the functor $\ilink{-}: \itop \to \topp$ commutes with homotopy pullbacks.
\end{lemma}

\begin{proof}
    By Proposition 2.3 of \cite{sayisvtelm}, $\ilink{-}$ is a right adjoint, and therefore commutes with limits, and in particular with pullbacks.
    Thus the functor $\ilink{-}$ commutes with pullbacks  for all strictly increasing chains of subgroups  $\mathbf{H}$. That is, if $i: X \to W$ and $j: Y \to W$ are isovariant maps, then $\ilink{X \times_W Y}\cong \ilink{X} \times_{\ilink{W}} \ilink{Y}$. 

    \medskip
    Using Proposition \ref{prop-cotensoring} with $A=[0,1]$ and $X=\Delta^\mathbf{H}$, 
    the functor $\ilink{-}$ commutes with path spaces for all strictly increasing chains of subgroups $\mathbf{H}$. That is, $\ilink{\ipathW} \cong \map([0,1], \ilink{W})$ for all $W$. 
\end{proof}

\begin{defn}\label{def-hopb-sq}
A commutative square in $\itop$ will be called a \textit{homotopy pullback square} if the induced map from the initial object of the square to the homotopy pullback of the rest of the square is an isovariant weak equivalence.
\end{defn}

    We will now discuss homotopy pushouts. As in \cite{sayisvtelm}, we use the double mapping cylinder for the homotopy pushout. That is, given the diagram
    \[
\xymatrix{Z \ar[r]^v \ar[d]^u & Y \\ X}
    \]
    the homotopy pushout is given by \[
    X \cup_u (Z \times [0,1]) \cup_v Y.
    \]

\begin{defn}
    Dually to a homotopy pullback square, a commutative square in $\itop$ will be called a \textit{homotopy pushout square} if the induced map from the homotopy pushout of the punctured square to the final object is an isovariant weak equivalence.
\end{defn}

For the higher Blakers--Massey theorem, we will need the notion of a homotopy (co)cartesian diagram for higher dimensional cubes. We follow the standard double cobar construction definitions for homotopy limits used in the equivariant setting, as in \cite[Ch V]{alaskanotes}. 

Recall from Proposition \ref{prop-cotensoring} that $\itop$ is cotensored over $\topp$. 
This implies that $\itop$ is also cotensored over simplicial sets. For a simplicial set $A_\bullet$ and an isovariant space $X$, define $X^{A_\bullet} = X^{|A_\bullet|}$. In order to study homotopy limits in the isovariant category, it will be useful to show that $\ilink{-}$ commutes with this cotensoring.

\begin{lemma}\label{lem-ilink-commutes-cotensoring}
    For a space $A$ and an isovariant space $X$, there is a canonical homeomorphism $\ilink{X^A} \cong \map(A, \ilink{X})$. Similarly, for a simplicial set $A_\bullet$, there is a canonical homeomorphism $\ilink{X^{A_\bullet}} \cong \map(|A_\bullet|, \ilink{X})$. 
\end{lemma}

\begin{proof} 
    This follows from Proposition \ref{prop-cotensoring}, replacing $A$ with $|A_\bullet|$ for the second statement.
\end{proof}

Recall from Definition \ref{def-iproduct} that the isovariant category $\itop$ has products; without the formal terminal object, this is limited to nonempty products, with the formal terminal object giving the empty product. 
The isovariant category also has equalizers, which agree with the equalizers on underlying spaces.
Since the isovariant category has products and equalizers and is cotensored over simplicial sets, we can define the totalization of a cosimplicial object in $\itop$.

\begin{defn}
    Let $X^\bullet : \Delta \to \itop$ be a cosimplicial isovariant space. Define its totalization $\Tot(X^\bullet)$ to be the equalizer of the two maps
    \[
    \xymatrix{
    \iprod \limits_n (X^n)^{\Delta^n} \ar@<1ex>[r] \ar@<-1ex>[r] & \iprod \limits_{\phi: [n] \to [k] \in \Delta} (X^k)^{\Delta^n}
    }
    \]
    The projections of the two maps to the $\phi$th factor are given by precomposition with the map $\Delta^n \to \Delta^k$ induced by $\phi: [n] \to [k]$, and postcomposition with the map $X^n \to X^k$ induced by $\phi$. 
\end{defn}

An important property of isovariant totalization is that it commutes with $\ilink{-}$.

\begin{prop}\label{prop-ilink-commutes-tot}
    For a cosimplicial isovariant space $X^\bullet$, there is a canonical homeomorphism
    $$\ilink{\Tot(X^\bullet)} \cong \Tot(\ilink{X^\bullet})$$
\end{prop}

\begin{proof}
    By \cite{sayisvtelm}[Prop 2.3], $\ilink{-}$ is a right adjoint, and therefore commutes with all limits. By Lemma \ref{lem-ilink-commutes-cotensoring}, it also commutes with the cotensoring of $\itop$ over $\topp$. Therefore it commutes with totalization, as required.
\end{proof}

We now define the cosimplicial replacement of a diagram in $\itop$ in the standard way.   

\begin{defn}\label{def-crep}
    Let $J$ be a small category, and $D: J \to \itop$ a functor. Define a cosimplicial isovariant space, $\crep(D): \Delta \to \itop$, as
    \[\crep(D)^n = \iprod_{j_0 \to \cdots \to j_n} D(j_n)\]
    where the isovariant product is indexed over all composable sequences $j_0 \to \cdots \to j_n$ of $n$ arrows in $J$. The cofaces are induced by composition in $J$ and evaluation $J(j_n,j_{n+1}) \times D(j_n) \to D(j_{n+1})$ and codegeneracies are induced by inserting identities. More explicitly, we describe the coface maps into the factor $D(j_{n+1})$ labeled by the chain $\nu = [j_0 \to \cdots \to j_{n+1}]$ in $\crep(D)^{n+1}$. For $0 \leq k \leq n$, we can ``cover'' $j_k$, resulting in the new chain $[j_0 \to \cdots \to \hat{j_k} \to \cdots \to j_{n+1}]$, which labels a factor of $D(j_{n+1})$ in $\crep(D)^{n}$, and $d^k:\crep(D)^n \to \crep(D)^{n+1}$ is defined by projecting to that factor, then mapping by the identity to the factor $D(j_{n+1})$ labeled by $\nu$ in $\crep(D)^{n+1}$, for $0 \leq k \leq n$. The map $d^{n+1}:\crep(D)^n \to \crep(D)^{n+1}$ composes the projection map with evaluating $D$ along the deleted morphism $D(j_n) \to D(j_{n+1})$ to the factor labeled by $\nu$. The codegeneracy maps $s^k: \crep(D)^{n+1} \to \crep(D)^n$ for $0 \leq k \leq n$ are defined by inserting identities. Explicitly, $s^k$ maps into the factor $D(j_n)$ labeled by $[j_0 \to \cdots \to j_n]$ by projecting $\crep(D)^n$ to the identical factor labeled by the chain in which one has inserted the identity map $j_k \to j_k$.
\end{defn}

As in \cite[\S 5.7]{Dug}, we can now use the cosimplicial replacement and the totalization to define homotopy limits.

\begin{defn}\label{def-holim}
    Let $J$ be a small category, and $D: J \to \itop$ a functor. Define
    $$\mathrm{holim}_J D : = \Tot(\crep(D)).$$
\end{defn}

One can check that for a pullback diagram, this agrees with the homotopy pullback $P(i,j)$ defined above.

\begin{prop}\label{prop-ilink-holim}
    The functor $\ilink{-}$ commutes with homotopy limits. That is, for $D: J \to \itop$, there is a canonical homeomorphism
    $$\ilink{\mathrm{holim}_J D} \cong  \mathrm{holim}_J(\ilink{D}) .$$
\end{prop}

\begin{proof}
    By Proposition \ref{prop-ilink-commutes-tot}, $\ilink{-}$ commutes with totalization. It also commutes with products, and therefore it commutes with $\crep$. Thus 
    $$\ilink{\mathrm{holim}_J D} =  \ilink{\Tot(\crep(D))}  \cong \Tot(\crep(\ilink{D})) = \mathrm{holim}_J(\ilink{D})$$
    as required.
\end{proof}

\begin{defn}\label{def-n-cube}
    An $S$-cube in a category $\cC$ is a functor $\mathcal{X}:\mathcal{P}(S) \to \cC$,
where $S$ is a finite set and $\mathcal{P}(S)$ is the poset of all subsets of $S$. When $S = \underline{n} = \{1, 2, \dots , n\}$, we will refer to the resulting $S$-cube as an $n$-cube. Note that these are strictly commutative $n$-cubes.
\end{defn}

Let $\mathcal{P}_0(\underline{n})$ denote the poset of nonempty subsets of $\underline{n}$.
We call an $n$-cube $\mathcal{X}$ \textit{strongly cocartesian} if each face of dimension $2$ is a homotopy pushout square.

\section{The Blakers--Massey theorem}
We now prove three versions of the Blakers--Massey theorem, using methods of diagram categories. First, we rephrase and reprove \cite[2.1]{Hau}, an equivariant version. Then we prove the analogous isovariant version and its $n$-cubical generalization. 

\subsection{The equivariant Blakers--Massey theorem}
Let $G$ be a finite group. We recall the definition of equivariant connectivity for a $G$-equivariant map.

\begin{defn} Let $n_\bullet: Sub(G) \to \mathbb{Z}$ be a conjugacy-invariant function on the set of subgroups of $G$, and let $f:X \to Y$ be a $G$-map.
  We say that $f$ is $n_\bullet$-connected if $f^H: X^H \to Y^H$ is an $n_H$-connected map for all $H \leq G$. 
  \end{defn}

  An equivalent formulation of the definition is that $f$ is $n_\bullet$-connected if the induced map $\emap(G/H,f): \map(G/H,X) \to \emap(G/H,Y)$ is $n_H$-connected for all subgroups $H \leq G$. We generalize this to the isovariant setting in Definition \ref{def-isvtconn}.

As preparation for the isovariant Blakers--Massey theorem, we reprove Hauschild's equivariant Blakers--Massey theorem, which we rephrase below. In the equivariant category, a homotopy pushout square of $G$-spaces is a commutative square of equivariant maps such that the induced map from the homotopy pushout of the punctured square to the final object is an equivariant weak equivalence.

\begin{thm}\label{thm-EBM} {(Theorem 2.1 of \cite{Hau})}
    Suppose that the following is a homotopy pushout square of $G$-spaces:
    \[\xymatrix{
Z \ar[r]^-v \ar[d]_-u & Y \ar[d]^-j \\
X \ar[r]^-i & W
}\]
If $u$ is $n_\bullet$-connected and $v$ is $m_\bullet$-connected, then the (homotopy) cartesian gap map of the square is $(n_\bullet + m_\bullet - 1)$-connected.
\end{thm}

\begin{proof}
    Homotopy pushouts commute with fixed points of a finite group action, and therefore the following is a homotopy pushout square of spaces for all $H \leq G$.
    \begin{gather}
\begin{aligned}
\xymatrix{
Z^H \ar[r]^-{v^H} \ar[d]_-{u^H} & Y^H \ar[d]^-{j^H} \\
X^H \ar[r]^-{i^H} & W^H
}
\end{aligned}
\label{eqvt-fixed-square}
\end{gather}
We will now identify the cartesian gap map of this square of spaces. We will use the following model for the homotopy pullback of $i$ and $j$
\[P = X \times_W \map(I,W) \times_W Y = \{ (x, \gamma, y) : \gamma(0) = i(x), \gamma(1) = j(y) \}.\]

Then $P^H$ is the homotopy pullback of $i^H$ and $j^H$. Furthermore, consider the cartesian gap map $c: Z \to P$, given by $c(z) = (u(z), const_{(i \circ u) (z)}, v(z))$, where we use the fact that $i \circ u = j \circ v$.
 Here $const_w$ is the path in $W$ given by $const_w(t) = w$. 
Then $c^H: Z^H \to P^H$ is the cartesian gap map of the square (\ref{eqvt-fixed-square}) above.

By assumption, $u^H$ is $n_H$-connected and $v^H$ is $m_H$-connected. By the ordinary Blakers--Massey theorem, $c^H$ is therefore $(n_H + m_H - 1)$-connected. Therefore $c$ is $(n_\bullet + m_\bullet - 1)$-connected, as required.
\end{proof}

Note that Hauschild's result states a range in which the map
$\pi_V(Y,Z) \to \pi_V(W,X)$
is bijective or surjective for a $G$-representation $V$. This follows from Theorem \ref{thm-EBM} and properties of representation-graded homotopy groups.

\subsection{The isovariant Blakers--Massey Theorem}
We will now follow the blueprint of the proof of Theorem \ref{thm-EBM} to prove an isovariant Blakers--Massey theorem, using the fact from Lemma \ref{lem-commutes-hopb} that the functor $\ilink{-}=\imap(\Delta^\mathbf{H}, -): \itop \to \topp$ commutes with homotopy pullbacks.

We first define isovariant connectivity. For an isovariant map $f: X \to Y$, denote by $\ilink{f}$ the induced map $\imap(\Delta^\mathbf{H}, f): \imap(\Delta^\mathbf{H}, X) \to \imap(\Delta^\mathbf{H}, Y)$.

\begin{defn} \label{def-isvtconn} Let $n_\bullet: \{ \text{chains of strictly increasing subgroups of $G$} \} \to \mathbb{Z}$ be a conjugacy-invariant function on the set of chains of subgroups of $G$, and let $f:X \to Y$ be an isovariant map. We say that $f$ is $n_\bullet$-connected if $\ilink{f}: \ilink{X} \to \ilink{Y}$ is an $n_\mathbf{H}$-connected map of spaces for all chains of strictly increasing subgroups $\mathbf{H}$.
\end{defn}

There are some important differences between isovariant connectivity and equivariant connectivity. For example, a finite-dimensional $G$-representation $V$ is not isovariantly weakly contractible. That is, the unique map from $V$ to the formal terminal object is not an isovariant weak equivalence. Moreover, the inclusion $V \to S^V$ has isovariant connectivity $n_\mathbf{H} = \infty$ for all chains of subgroups $\mathbf{H}$ that do not contain $G$.

\medskip

We are ready to prove the isovariant Blakers--Massey theorem. We note that this is a special case of Theorem \ref{thm-higher-isvt-BM}. 

\begin{thm}\label{thm-IBM}
    Suppose that the following is a homotopy pushout square in the isovariant category:
    \[\xymatrix{
Z \ar[r]^-v \ar[d]_-u & Y \ar[d]^-j \\
X \ar[r]^-i & W
}\]
If $u$ is $n_\bullet$-connected and $v$ is $m_\bullet$-connected, then the cartesian gap map of the square is $(n_\bullet + m_\bullet - 1)$-connected.
\end{thm}

\begin{proof}
    By Lemma 3.2 of \cite{sayisvtelm}, $\ilink{-}$ commutes with homotopy pushouts up to homotopy. Therefore the square
\begin{gather}
\begin{aligned}
\xymatrix{
\ilink{Z} \ar[r]^-{\ilink{v}} \ar[d]_-{\ilink{u}} & \ilink{Y} \ar[d]^-{\ilink{j}} \\
\ilink{X} \ar[r]^-{\ilink{i}} & \ilink{W}
}
\end{aligned}
\label{M-square}
\end{gather}
is a homotopy pushout square of spaces for all $\mathbf{H}$.

Let $P(i,j)$ be the homotopy pullback of $i$ and $j$ defined in Definition \ref{def-pij}. Then by Lemma \ref{lem-commutes-hopb}, $\ilink{P(i,j)}$ is the homotopy pullback of $\ilink{i}$ and $\ilink{j}$.  Furthermore, consider the cartesian gap map $c: Z \to P(i,j)$, given by $c(z) = (u(z), const_{(i \circ u) (z)}, v(z))$, where we use the fact that $i \circ u = j \circ v$.
 Here $const_w$ is the path in $W$ given by $const_w(t) = w$.

Then $\ilink{c}: \ilink{Z} \to \ilink{P(i,j)}$ is the cartesian gap map of the square (\ref{M-square}) above.

By assumption, $\ilink{u}$ is $n_\mathbf{H}$-connected and $\ilink{v}$ is $m_\mathbf{H}$-connected. By the ordinary Blakers--Massey theorem, $\ilink{c}$ is therefore $(n_\mathbf{H} + m_\mathbf{H} - 1)$-connected. Therefore $c$ is $(n_\bullet + m_\bullet - 1)$-connected, as required.
\end{proof}

Analogous to the ordinary Blakers--Massey theorem, the isovariant version does not give the best estimate of the connectivity of the cartesian gap map in all cases, but does give the optimal connectivity in at least one case. 

\begin{ex}
For a finite dimensional real representation $V$, let $S(V)$ and $D(V)$ represent the unit sphere and unit disk, respectively. The one point compactification of $V$, denoted $S^V$, is also given by the equivariant quotient $S(V)_+ \to D(V)_+ \to S^V$. Isovariantly, the notion of a quotient does not make sense; any pushout involving the formal terminal object results in the formal terminal object. Thus we think of $S^V$ as given by the pushout (both equivariantly and isovariantly) of two copies of $S(V) \to D(V)$.
\begin{itemize}
    \item  Let $\sigma$ represent the sign representation of $C_2$. One can check that the isovariant pullback of 
    \[
    \xymatrix{& D(\sigma) \ar[d] \\ D(\sigma) \ar[r] & S^\sigma}
    \]
    is isovariantly equivalent to $S(\sigma)$, so the cartesian gap map $c$ is an equivalence. Note that the isovariant Blakers--Massey theorem only guarantees that it will be an equivalence on $\{e\}$-links (that is, that $\imap(\Delta^{\{e\}},c)$ is an equivalence), since $S(\sigma) \to D(\sigma)$ has $n_e = \infty$, $n_G = -1$, and $n_{e<G} = -1$.
    \item Let $\rho$ represent the regular representation of $C_2$. Then $S(\rho) \to D(\rho)$ has $n_e = \infty$, $n_G = 0$, and $n_{e<G} = 0$. The pushout of $S(\rho) \to D(\rho)$ along $S(\rho)\to D(\rho)$ is $S^\rho$, and the cartesian gap map has $n_e = \infty$, $n_G = -1$, and $n_{e < G} = -1$. In this example, the connectivity bound obtained using the isovariant Blakers--Massey theorem is the best possible.
    \item  One can verify that the cartesian gap map for $S(\sigma \oplus \sigma) \to D(\sigma \oplus \sigma)$ is an equivalence, even though $S(\sigma \oplus \sigma) \to D(\sigma \oplus \sigma)$ is not an equivalence.
\end{itemize}
   
\end{ex}

\subsection{A higher isovariant Blakers-Massey theorem}
The Blakers--Massey theorem was generalized to $n$-cubical diagrams by Ellis-Steiner and Goodwillie (\cite{ES}, \cite{G2}), with applications in Goodwillie's calculus of functors. We prove an isovariant version.

\begin{thm} \label{thm-higher-isvt-BM}
    Let $\mathcal{X}$ be a strongly cocartesian $n$-cube in $\itop$, with $n \geq 2$. Suppose $\mathcal{X}(\emptyset) \to \mathcal{X}(\{s\})$ is $(k_s)_\bullet$-connected, for each $s \in \underline{n}$. Then $\mathcal{X}$ is $k_\bullet$-cartesian, with $k_\bullet=1-n+\sum_s (k_s)_\bullet$. That is, the cartesian gap map $\mathcal{X}(\emptyset) \to \hl_{\mathcal{P}_0(\underline{n})} \mathcal{X}$ is $k_\bullet$-connected.
\end{thm}

\begin{proof}
    Let $\cX$ be a strongly cocartesian $n$-cube in $\itop$. Since $\ilink{-}$ commutes with homotopy pushouts up to homotopy, the $n$-cube $\ilink{\cX}$ is strongly cocartesian in $\topp$ for each strictly increasing chain of subgroups $\mathbf{H}$. Each initial map $\ilink{\cX(\emptyset) \to \cX(\{s\})}$ is $(k_s)_\mathbf{H}$-connected, so the higher Blakers--Massey theorem in $\topp$ (see, e.g., Theorem 2.3 of \cite{G2}) tells us that $\ilink{\cX}$ is $k_\mathbf{H}$-cartesian, where $k_\mathbf{H}=1-n+\sum_s (k_s)_\mathbf{H}$. That is, the cartesian gap map $\ilink{\cX(\emptyset)} \to \hl_{\cP_0(\underline{n})} \ilink{\cX}$ is $k_\mathbf{H}$-connected. By Proposition \ref{prop-ilink-holim}, the mapping space $\ilink{-}$ commutes with homotopy limits, so $\hl_{\cP_0(\underline{n})} \ilink{\cX} \simeq \ilink{\hl_{\cP_0(\underline{n})} \cX}$. Thus the map $\ilink{\cX(\emptyset)} \to \ilink{\hl_{\cP_0(\underline{n})} \cX}$ is $k_\mathbf{H}$-connected, so $\cX$ satisfies the definition of being $k_\bullet$-cartesian. 
\end{proof}

\section{Isovariant suspension}
In classical homotopy theory, the Blakers--Massey theorem is used to prove the Freudenthal suspension theorem. In this section, we will prove an isovariant Freudenthal suspension theorem. First, we will define isovariant suspension. In classical homotopy theory, the suspension of a space $X$ can be defined as a homotopy pushout of the diagram
$$\xymatrix{
X \ar[r] \ar[d] & {*} \\
{*}
}$$
We cannot mimic this definition directly in the isovariant category; the only space which receives a map from the formal terminal object is the formal terminal object itself, so the colimit of any pushout diagram involving the formal terminal object will be the formal terminal object.

\medskip

Instead, we will make use of a homotopy terminal object: an isovariant space which is isovariantly weakly equivalent to the formal terminal object. Our candidate for this homotopy terminal object will be a complete $G$-universe.

\medskip

\begin{defn}\label{def-complete-universe}
    Let $V_1, V_2, .., V_k$ be representatives for all the irreducible $G$-representations. A complete $G$-universe $\mathcal{U}$ is any $G$-representation isomorphic to a direct sum of countably many copies of each $V_i$:
$$\mathcal{U} \cong V_1^{\oplus \mathbb{N}} \oplus V_2^{\oplus \mathbb{N}} \oplus ... \oplus V_k^{\oplus \mathbb{N}}$$
If $\rho$ is the regular representation of $G$, $\mathcal{U}$ is also isomorphic to $\D{\rho ^{\oplus \mathbb{N}} = \mathop{\mathrm{colim}}\limits_{n \in \mathbb{N}} \rho^{\oplus n}}$. 
\end{defn}

\begin{rem} \label{Uisnice}
    As a sequential colimit of $G$-manifolds along smooth embeddings, $\mathcal{U}$ is built out of $\mathcal{I}$-cells, and therefore the map $\emptyset \to \mathcal{U}$ is in $\icell \subset \icof$. The map from $\mathcal{U}$ to the terminal object is in $\jinj$ (as in the representation case of Theorem 3.9 of \cite{KYfixed}). Thus $\mathcal{U}$ is both $\mathcal{J}$-fibrant and $\mathcal{I}$-cofibrant.
\end{rem}

In the next theorem, we will compare $\mathcal{U}$ to an $\mathcal{I}$-cofibrant and $\mathcal{J}$-fibrant isovariantly contractible space.

\begin{defn}
    Denote by $\mathcal{V}$ an ``$\mathcal{I}$-cofibrant replacement" of the formal terminal object. This is obtained by factoring the unique map from the empty set to the formal terminal object as a map in $\icof$ followed by a map in $\iinj$.
\end{defn}

\begin{rem} \label{Visnice}
The space $\mathcal{V}$ is isovariantly weakly contractible because 
the map from $\mathcal{V}$ to the terminal object is in $\iinj \subset \mathcal{W}$ (see Lemma 3.2 of \cite{KYfixed}). We note that $\mathcal{V}$ is also both $\mathcal{J}$-fibrant and $\mathcal{I}$-cofibrant, since $\iinj \subset \jinj$ (see the proof of Theorem 5.3 in \cite{KYfixed}).
\end{rem}

\begin{thm}\label{prop-h-terminal}
    A complete $G$-universe $\mathcal{U}$ is isovariantly weakly contractible. That is, for all chains of subgroups $\mathbf{H}$, $\ilink{\mathcal{U}}$ is weakly contractible.
\end{thm}

\begin{proof} 
   
    By Remarks \ref{Uisnice} and \ref{Visnice} above, there is a lift of $\emptyset \to \mathcal{V}$ to $\mathcal{U} \to \mathcal{V}$; we claim that it is an isovariant weak equivalence. 
   
    As in Proposition 3.12 of \cite{KYfixed}, since both $\mathcal{U}$ and $\mathcal{V}$ are $\mathcal{J}$-fibrant and $\mathcal{I}$-cofibrant, it's enough to show that the map $\mathcal{U} \to \mathcal{V}$ induces weak equivalences on 0 and 1 dimensional links. Since $\mathcal{V}$ is isovariantly weakly contractible, we need to show that:
    \begin{enumerate}
        \item $\mathcal{U}_H$ is weakly  contractible for all subgroups $H \leq G$, and
        \item $\imap(\Delta^{H < K}, \mathcal{U})$ is weakly contractible for all subgroups $H < K \leq G$, where $$\imap(\Delta^{H<K}, \mathcal{U}) = \{ \gamma: [0,1] \to \mathcal{U} : \gamma(0) \in \mathcal{U}_K, \gamma(s) \in \mathcal{U}_H  \forall s>0 \}.$$
    \end{enumerate}
    The first follows from Proposition 4.6 of \cite{MMhcob} with $X=\ast$. For the second, we will follow the proof strategy of Proposition 4.6 of \cite{MMhcob}, by constructing an explicit trivializing extension of any map of a sphere into $\imap(\Delta^{H<K},\mathcal{U})$. Denote by $\rho$ the regular representation of $G$, with basis vectors $\chi_g$ for $g \in G$, satisfying $g'\chi_g = \chi_{g'g}$. For every subgroup $H \leq G$,  the element $\sum\limits_{h \in H} \chi_h$ is fixed by exactly $H$; the sum is rearranged when acted on by $g \in H$, but action by $g \not \in H$ introduces a summand $\chi_{gh}$, which is not in the original sum.
     
    For any subgroups $H<K$ of $G$, we can define a path $\gamma:[0,1] \to \rho$ by
    \[
    \gamma(s)= s\sum\limits_{h \in H} \chi_h + (1-s) \sum\limits_{k \in K} \chi_k.
    \] 
    By the preceeding argument, $\gamma(0) \in \rho_K$, and we will show that when $s > 0$, $\gamma(s) \in \rho_H$. Action by $g \in H$ fixes both sums since $H<K$. We can rewrite $\gamma(s)=\sum\limits_{h \in H} \chi_h + (1-s) \sum\limits_{k \in K-H}  \chi_k$.
    If $g \in K-H$, action by $g$ fixes the second sum but not the first. If $g \notin K$, the action produces a linear combination of $\chi_{g'}$ for $g' \notin K$, thus $\gamma(s)$ is not fixed by $g \notin H$.
    Using the $G$-action on $\rho$, one can extend $\gamma$ to an element of $\imap(\Delta^{H<K}, \rho)$. 

    \medskip

    We will show that for all $H <K$ and every $n$, there is a lift in every diagram
$$\xymatrix{
S^{n-1} \ar[r]^-f \ar[d] & \imap(\Delta^{H<K}, \mathcal{U}) \\
D^n 
}$$
Recall that $\D{\mathcal{U} = \col \limits_{l \in \mathbb{N}} \rho^{ \oplus l}}$, and therefore for every $l$, we have maps $\imap(\Delta^{H<K}, \rho^{\oplus l}) \to \imap(\Delta^{H<K}, \rho^{\oplus l+1}) \to \imap(\Delta^{H<K}, \mathcal{U})$, induced by the inclusions $\rho^{\oplus l} \subset \rho^{\oplus l+1} \subset \mathcal{U}$. Recall that
$$\imap(\Delta^{H<K}, \mathcal{U}) = \{ \gamma: [0,1] \to \mathcal{U} : \gamma(0) \in \mathcal{U}_K, \gamma(s) \in \mathcal{U}_H  \forall s>0 \}.$$
Adjointing, a map $S^{n-1} \to \imap(\Delta^{H<K}, \mathcal{U})$ gives a map $S^{n-1} \times [0,1] \to \mathcal{U}$ satisfying certain conditions (using a fundamental domain for $\Delta^{H<K}$). Since $S^{n-1} \times [0,1]$ is compact, any such map factors through a finite stage of the colimit, which still satisfies the same conditions making the adjoint map isovariant. Therefore $f$ factors as a map $S^{n-1} \to \imap(\Delta^{H<K}, \rho^{\oplus l})$ for some $l \in \mathbb{N}$. We will show that there is a lift $D^n \to \imap(\Delta^{H<K}, \rho^{\oplus l+1})$ in the diagram
$$\xymatrix{
S^{n-1} \ar[r]^-f \ar[d] & \ar[r]  \imap(\Delta^{H<K}, \rho^{ \oplus l})  \ar[r]  &\imap(\Delta^{H<K}, \rho^{\oplus l + 1}) \\
D^n
}$$
and composing with the map $\imap(\Delta^{H<K},\rho^{\oplus l+1}) \to \imap(\Delta^{H<K}, \mathcal{U})$ induced by the inclusion, this will give a lift of $f: S^{n-1} \to \imap(\Delta^{H<K}, \mathcal{U})$ to a map $D^n \to \imap(\Delta^{H<K}, \mathcal{U})$.

\medskip

Denote $V = \rho^{\oplus l}$, so that $\rho^{\oplus l+1} = V \times \rho$. 

For $y \in D^n$, denote $t_y = 1 - ||y||$. If $y \neq 0$, denote $y_0 = \frac{y}{||y||}$, and note that for $y \in S^{n-1}$, $y_0 = y$. Now we define a map $\tilde{f}: D^n \times [0,1] \to V \times \rho$ and check that its adjoint produces the desired lift $\tilde{f}:D^n \to \imap(\Delta^{H<K}, V \times \rho)$. Let

\[
\tilde{f}(y)(s)= \begin{cases}
    (f(y_0)(s), 2t_y \gamma(s) ) & \text{ if } t_y \leq 1/2\\
    ( (2-2t_y)f(y_0)(s), \gamma(s) ) & \text{ if } 1/2 \leq t_y < 1\\
    (0, \gamma(s)) & \text{ if } t_y=1 \text{ i.e. } y=0
\end{cases}
\]

\medskip

We now check that this satisfies all of the conditions we need. For $G$-spaces $X$ and $Y$, $X_H \times Y^H \subseteq (X \times Y)_H$.

\begin{itemize}
    \item This is well-defined and continuous. If $t_y = 1/2$, the first and the second expressions above agree. As $t_y \to 1$ (i.e. as $y \to 0$), $(2-2t_y)f(y_0)(s) \to 0$, so the second expression approaches the third expression.
    \item This extends $f$. If $y \in S^{n-1}$, then $t_y = 0$ and $y_0 = y$, so that we get $\tilde{f}(y)(s) = (f(y)(s), 0)$.
    \item For $s=0$, $\tilde{f}(y)(s) \in (V \times \rho)_K$. Because if $s=0$ and $t_y \leq 1/2$, the first coordinate of $\tilde{f}(y)(0)$ is $ f(y_0)(0) \in V_K$, and the second coordinate $2t_y \sum\limits_{k \in K} \chi_k$ is certainly fixed by $K$. Similarly for $1/2 \leq t_y <1$. If $t_y = 1$,  the second coordinate of $\tilde{f}(y)(0)$ is $\sum\limits_{k \in K} \chi_k \in \rho_K$, and the first coordinate 0 is certainly fixed by $K$.
    \item For $s >0$, $\tilde{f}(y)(s) \in (V \times \rho)_H$. Because if $t_y < 1$, the first coordinate is in $V_H$ by assumption on $f$, and the second coordinate is certainly fixed by $H$. If $t_y =1$, the second coordinate is in $\rho_H$, and the first coordinate 0 is certainly fixed by $H$.
\end{itemize}

Therefore we have defined a lift, so that $\mathcal{U}$ is isovariantly weakly contractible, as required.
\end{proof}

By the isovariant Whitehead theorem (Theorem 3.10 of \cite{KYfixed}), $\mathcal{U}$ and $\mathcal{V}$ are isovariantly homotopy equivalent, not just isovariantly weakly equivalent. We obtain:

\begin{cor}\label{cor-map-to-complete}
    Suppose that $X$ is $\mathcal{I}$-cofibrant. Then there is an isovariant map $X \to \mathcal{U}$, unique up to isovariant homotopy.
\end{cor}

\begin{proof}
     
    Since $\emptyset \to X \in \icof$ and the map from $\mathcal{U}$ to the terminal object is in $\jinj \cap \mathcal{W}\subseteq \iinj$ (see Lemma 3.4 of \cite{KYfixed}), we indeed get a lift $f: X \to \mathcal{U}$. Now suppose that $f': X \to \mathcal{U}$ is any other isovariant map. We want to show that there is an isovariant lift in the diagram
    $$\xymatrix{
    X \amalg X \ar[r]^-{f \amalg f'} \ar[d] & \mathcal{U} \\
    X \times [0,1]
    }$$
    We note that if a map $\phi\in \icof$, then $\phi \Box s_0 \in \icof$, since $s_k \Box s_0 \cong s_{k+1}$.
    
    Since $\emptyset \to X \in \icof$, the map $X \amalg X \to X \times [0,1]$ is in $\icof$ as well.  The map from $\mathcal{U}$ to the terminal object is in $\iinj$, and therefore there is an isovariant lift in the diagram, as required. 
\end{proof}

The following measures the connectivity of maps to $\mathcal{U}$.

\begin{prop}\label{prop-conn-to-terminal}
    Suppose that $X$ is isovariantly $n_\bullet$-connected (that is, every $\ilink{X}$ is $n_{\mathbf{H}}$-connected.) Then any isovariant map $\iota: X \to \mathcal{U}$ is $(n_\bullet + 1)$-connected.
\end{prop}

\begin{proof}
    By Theorem \ref{prop-h-terminal}, $\ilink{\mathcal{U}}$ is weakly contractible for all $\mathbf{H}$. If $\pi_i(\ilink{X}) = 0$ for $i \leq n_\mathbf{H}$, then $\ilink{\iota}: \ilink{X} \to \ilink{\mathcal{U}}$ is therefore an isomorphism on $\pi_i$ for $i \leq n_\mathbf{H}$ and surjective on $\pi_i$ for $i = n_\mathbf{H} + 1$.
\end{proof}

\begin{defn}
    Let $X$ be $\mathcal{I}$-cofibrant, and let $\iota: X \to \mathcal{U}$ be an isovariant map. We will call the homotopy pushout $\mathcal{U} \cup_\iota (X \times [0,1]) \cup_\iota \mathcal{U}$ an isovariant suspension of $X$, and denote it $\isusp X$. This is isovariantly homeomorphic to the pushout $C_\mathcal{U}X  \cup_X  C_\mathcal{U} X $, where $C_\mathcal{U} X= \mathcal{U} \cup_\iota (X \times [0,1])$ is the mapping cone of $\iota$.
\end{defn}

Note that this is an unreduced suspension. Denote the unreduced suspension in topological spaces by $S$. Since $\ilink{-}$ commutes with homotopy pushouts up to homotopy, we have $\ilink{\isusp X} \simeq S \ilink{X}$.
We also note that the isovariant suspension is defined only for $\mathcal{I}$-cofibrant objects, or more generally for $G$-spaces equipped with an isovariant map to $\mathcal{U}$.

\begin{defn}
    Let $Y$ be a $G$-space with isovariant maps $j_0, j_1: \mathcal{U} \to Y$ (e.g. $Y = \isusp X$.) We will call the homotopy pullback $\mathcal{U} \times_{Y}  Y^{[0,1]} \times_Y \mathcal{U}$ of the diagram
$$\xymatrix{
\iloop Y \ar[r] \ar[d] & \mathcal{U} \ar[d]^-{j_0} \\
\mathcal{U} \ar[r]^-{j_1} & Y
}$$
the isovariant loop space of $Y$ and denote it $\iloop Y$. When $j_0$ and $j_1$ are isovariantly homotopic, we have $\ilink{\iloop Y} \simeq \Omega \ilink{Y}$.
\end{defn}

\begin{rem}\label{rmk-cone-hopb}
    Note that when $Y = \isusp X$, the loop space $\iloop \isusp X$ is isovariantly homotopy equivalent to the homotopy pullback of
    $$\xymatrix{
 & C_\mathcal{U} X \ar[d] \\
C_\mathcal{U} X \ar[r] & \isusp X
}$$
This is because the inclusion $\mathcal U \to C_\mathcal{U} X$ is an isovariant homotopy equivalence.
\end{rem}

Since $\ilink{-}$ commutes with homotopy pullbacks, we have $\ilink{\iloop \isusp X} \simeq \Omega_{s_0, s_1} S\ilink{X}$. Here $\Omega_{s_0, s_1}$ denotes paths from one suspension point to the other in the unreduced suspension.

\medskip

We are now ready to prove an isovariant Freudenthal suspension theorem.

\begin{cor}\label{thm-isvt-Freudenthal}
    Let $X$ be isovariantly $n_\bullet$-connected and $\mathcal{I}$-cofibrant. Then the cartesian gap map $c: X \to \iloop \isusp X$ is isovariantly $(2n_\bullet + 1)$-connected.
\end{cor}

\begin{proof}
    By Proposition \ref{prop-conn-to-terminal}, any isovariant map $\iota: X \to \mathcal{U}$ is $(n_\bullet + 1)$-connected. The map $c$ is the cartesian gap map of the homotopy pushout square
    $$\xymatrix{
    X \ar[r] \ar[d] & C_\mathcal{U} X \ar[d] \\
    C_\mathcal{U} X \ar[r] & \isusp X
    }$$
    so, by Theorem \ref{thm-IBM}, $c$ is $(n_\bullet + 1) + (n_\bullet + 1) - 1 = (2n_\bullet + 1)$-connected.
\end{proof}

\begin{rem}
The isovariant suspension $S_\mathcal{U}$ corresponds to the suspension of a $G$-space by a trivial 1-dimensional representation. In future work, we will explore the possibility of suspending by other representations, as in the equivariant Freudenthal suspension theorem (see, e.g., Theorem XI.4.5 of \cite{alaskanotes}.)
\end{rem}

\begin{ex}
    Let $\imap(\Delta^{\mathbf{H}},X)$ be denoted $X_{\mathbf{H}}$ in this example.
\begin{itemize}
    \item Let $\sigma$ denote the 1-dimensional sign representation of $C_2$. Then $\sigma_{\{e\}} = \mathbb{R} - \{0\}$, $\sigma_{C_2} = \{0\}$, and $\sigma_{\{e \}<C_2} \simeq \mathbb{R} - \{ 0 \}$. Therefore $((\isusp)^k \sigma)_{\{e\}} \simeq ((\isusp)^k
    \sigma)_{\{e\}<C_2} \simeq  S^k$, and $((\isusp)^k \sigma)_{C_2} \simeq *$. 
    
    \item Taking the 1-point compactification of the previous example, $S^\sigma$ satisfies $(S^\sigma)_{\{e\}} = \mathbb{R} - \{0\}$ and $(S^\sigma)_{C_2} \cong S^0$. For $(S^\sigma)_{\{e \}<C_2}$, we note that this is a punctured tubular neighborhood of $(S^\sigma)_{C_2} = S^0$ inside $S^\sigma$, and therefore $(S^\sigma)_{\{e \}<C_2} \simeq  \coprod_4 \mathbb{R}$. Therefore $((\isusp)^k S^\sigma)_{\{e\}} \simeq S^k$, $((\isusp)^k S^\sigma)_{\{e\}<C_2} \simeq  \bigvee_3 S^k$, and $((\isusp)^k S^\sigma)_{C_2} \simeq S^k$.
    \item Let $\rho$ denote the regular representation of $C_3$. Then $\rho \cong \mathbb{R} \oplus \lambda$, where $\lambda$ denotes the 2-dimensional rotation by $2\pi/3$ representation. Taking 1-point compactification, $S^\rho_{C_3} = S^1$, and $S^\rho_{\{e\}} = S^3 - S^1 = \mathbb{R}^3 - (\mathbb{R} \times \{0\}) \simeq S^1$. 
     Also, $S^\rho_{\{e\} < C_3}$ is equivalent to a punctured tubular neighborhood of $S^1$ in $S^3$. A tubular neighborhood of this $S^1$ in $S^3$ is homeomorphic to $S^1 \times D^2$, so removing $S^1 \times \{ 0 \}$ from it gives $S^\rho_{\{e\} < C_3} \simeq S^1 \times S^1$. Thus $((\isusp)^n S^\rho)_{C_3} \simeq S^{n+1}$, $((\isusp)^n S^\rho)_{\{e \}} \simeq S^{n+1}$, and for $n \geq 1$, $((\isusp)^n S^\rho)_{\{e \} < C_3} \simeq S^{n+1} \bigvee S^{n+1} \bigvee S^{n+2}$. This is because $\Sigma(S^1 \times S^1)$ splits as $ S^2 \bigvee S^2 \bigvee S^3$.
\end{itemize}
\end{ex}

\begin{rem}
    Note that it is not generally true that $(\isusp)^n S^V$ is isovariantly weakly equivalent to $S^{n+V}$ (although they are equivariantly weakly equivalent). For example, take $V = \sigma$, the sign representation of $C_2$, and $n=1$. Then $(\isusp S^\sigma)_e \simeq S(S^\sigma)_e \simeq S^1$, whereas $(S^{1 + \sigma})_e = S^{1+ \sigma} - S^1 \simeq S^0$. Therefore $\isusp S^\sigma$ and $S^{1 + \sigma}$ cannot be isovariantly weakly equivalent.
\end{rem}

\bibliography{isvtbmbib}{}

\end{document}